    \documentclass[reqno]{amsproc}
    \usepackage{amsmath,amssymb,amsthm,amsfonts,latexsym}
   \usepackage{euscript,color,textcomp,mathrsfs}

    \theoremstyle{plain}
   \newtheorem{thm}{Theorem}
   
   \newtheorem{lem}[thm]{Lemma}

   \theoremstyle{definition}

   \theoremstyle{remark}
   \newtheorem{rem}[thm]{{\it Remark}}

   \def\tlabel{\label}

\DeclareMathOperator{\okr}{{\stackrel{{\scriptscriptstyle{\mathsf{def}}}}{=}}}
\DeclareMathOperator{\D}{d\!} 
   \DeclareMathOperator{\conv}{conv}

\def\funkk#1#2#3#4{#1\ni#2\mapsto#3\in#4}

\def\is#1#2{\langle#1,#2\rangle}
\def\Le{\leqslant}

\def\odw{^{-1}}
\def\res#1{|_{#1}}
\def\spek#1{{\rm sp}(#1)}
\def\zb#1#2{\{{#1}\colon\ {#2}\}}
\def\funk#1#2#3{#1\colon#2\to#3}

\def\aac{\mathcal A}
\def\bbc{\mathcal B}
\def\ccc{\mathcal C}
\def\hhc{\mathcal H}
\def\kkc{\mathcal K}
\def\mmc{\mathcal M}
\def\rrc{\mathcal R}

\def\xxc{\mathcal X}
\def\ccb{\mathbb C}
\def\aaf{\mathfrak A}
\def\wwf{\mathfrak W}
\def\bbs{\boldsymbol B}
\def\qqs{\boldsymbol Q}

\begin{document}


   \title[Decomposing numerical ranges  and spectral sets   ]{Decomposing numerical ranges  along with spectral sets   }
   \author[F.H. Szafraniec]{Franciszek Hugon Szafraniec}
   \address{Instytut        Matematyki,         Uniwersytet
   Jagiello\'nski, ul. \L ojasiewicza 6, 30 348 Krak\'ow, Poland}
\email{franciszek.szafraniec@im.uj.edu.pl}
   \thanks{This work was partially supported  by the MNiSzW grant N201 026 32/1350. The
author also would like to acknowledge an assistance of the EU
Sixth Framework Programme for the Transfer of Knowledge ``Operator
theory methods for differential equations'' (TODEQ) \#
MTKD-CT-2005-030042.}
   \subjclass{Primary 47A12, 47A25;  Secondary 47A60,   46J10}
   \keywords{$K$-spectral set, numerical range, function algebra,
set of antisymmetry, Gleason part, F. and M. Riesz theorem,
elementary spectral measure, similarity, orthogonal decomposition,
contraction, $\ccc_\rho$ operators}
   \begin{abstract}  This note is to indicate the new sphere of applicability
of the method developed by Mlak as well as by the author.
Restoring those ideas is summoned by current developments
concerning $K$-spectral sets on numerical ranges.
   \end{abstract}
   \maketitle

   The decomposition of numerical ranges the title refers to is,
see \cite[p. 42]{gust},
   \begin{equation} \label{1.23.10}
   \wwf(A\oplus B)=\conv(\wwf(A)\cup\wwf(B));
   \end{equation}
   it can be proved for any two Hilbert space operators $A$ and
$B$. The other decomposition is that of the spectrum of a function
algebra related to a Hilbert space operator. These are the two
leading topics of the current paper.

   \section*{The algebra $\rrc(X)$}
   Here $X$ stands always for a compact subset of $\ccb$ and
$\ccc(X)$ does for the algebra of all continuous functions on $X$,
with the supremum norm. Denote by $\rrc(X)$ the closure in
$\ccc(X)$ of the algebra of all rational functions with poles off
$X$. The {\em spectrum}\,\footnote{\;Pretty often it is know under
the name of the maximal ideal space of the algebra.} $\xxc$ of
$\rrc(X)$, that is the set of all the characters (=nontrivial
multiplicative functionals) of $\rrc(X)$, can be identified with
$X$ itself via the evaluation functionals (which are apparently
characters)
   \begin{equation} \label{1.27.10}
   \funkk X x {\chi_x}{\xxc_\aac},\quad \chi_x(u)\okr
u(x),\;u\in\rrc(X);
   \end{equation}
this mapping is in fact a homeomorphism with the topologies this
of $\ccb$ and that of the $*$--weak topology of $\rrc(X)'$, the
topological dual (see \cite[Proposition 6.28]{ion} for a direct
argument\,\footnote{\;From the monographs on function algebras, we
recommend \cite{gam} as the most accessible one, also for missing
here terminology; the source \cite{ion} has also to be
mentioned.}).

   A positive measure $\nu_x$ on $X$ is called a representing
measure of $x$ (which is the same, according to the above
identification, as the corresponding characters $\chi_x$)
   \begin{equation*}
   \chi_x(u)=u(x)=\int_X u \D \nu_x, \quad u\in\rrc(X).
   \end{equation*}
   The point mass at $x$ is apparently a representing measure of
$x$; if it is the only representing measure of $x$ then $x$ is
called a {\em Choquet point} of $\rrc(X)$. All the Choquet points
form the {\em Choquet boundary} $\varSigma$ of $\rrc(X)$ and its
closure is just the {\em Shilov boundary} $\varGamma$ of
$\rrc(X)$.

    \subsection*{Sets of antisymmetry}
   Call subset $A$ of $X$ a {\em set of antisymmetry} of a
   function algebra $\aac$ if, for $u$ in $\aac$, $u$ real valued
   on $A$ implies $u$ is constant on $A$. Let $u\res A$ be the
   restriction of $u$ to $A$, $\aac\res A \okr\zb{u\res A}{u\in
   \aac}$. Then we can restate Bishop's theorem \cite{bish} as
   \begin{thm}\label{1.24.10}
   Every set of antisymmetry of $\aac$ is contained in a maximal
   set of antisymmetry. The collection $\aaf$ of maximal sets of
   antisymmetry forms a pairwise disjoint, closed covering of $X$
   satisfying
   \begin{enumerate}
   \item[(a)] $u\in\ccc(X)$ and $u\res A\in \aac\res A$ for every
   $A\in\aaf$ imply $u\in\aac$; \item[(b)] $\aac\res A$ is closed
   in $\ccc(A)$, $A\in\aaf$.
   \end{enumerate}
   \end{thm}
   An addition information we are going to use here comes form
   \cite{gli} (see also \cite[Chapter II, Theorem 12.7]{gam} for
   more explicit exposition).
   \begin{thm}\label{2.24.10}
   Every maximal set of antisymmetry $A$ is an intersection of
   peak sets. $A$ is an intersection of peak set if and only if
   $\mu\perp \aac$ implies $\mu\res A\perp \aac$.
   \end{thm}

    \subsection*{Gleason parts}  Due to the identification established by
\eqref{1.27.10} $X$ considered as the spectrum of $\rrc(X)$
decomposes disjointly into the sets, called the {\em Gleason
parts} of $\rrc(X)$, according to the relation\,\footnote{\;A full
description of this relation, in the general case of a function
algebra, can be found for instance in \cite[Theorem 3.12]{ion},
see \cite{bish} and \cite{bea} for the master results.}
$\|\chi_{x_1}-\chi_{x_2}\|<2$, $x_1,x_2\in X$ with the norm being
that of $\rrc(X)'$; this relation turns out to be an equivalence,
cf. \cite{glea}. The crucial point is that if $x_1$ and $x_2$ are
in the same Gleason part then there are measures $\nu_1$ and
$\nu_2$ on $\varGamma$ representing these points and such that
$c\nu_2\Le\nu_1$ and $c\nu_1\Le\nu_2$ with some $0<c<1$, see
\cite{bish1}. This allows us to think of complex measures
absolutely continuous with respect to a Gleason part $G$, in short
{\em $G$--continuous}.

   Denote by $\mu_G$ a $G$--continuous part of $\mu$ with respect
   to a Gleason part $G$. Let $(G_\alpha)_\alpha$ be the
   collection of all Gleason parts of $\rrc(X)$. It is known
   \cite{gli1} (see also \cite[ChapterVI, Theorem 3.4]{gam}) that

   \begin{thm} \tlabel{t1.26.11}
   Suppose $X\subset \ccb$ is compact. Then
    \begin{enumerate}
   \item[(i)] for every $\mu\in\mmc(X)$ there exists a unique
   $\mu_0\in\mmc(X)$ and a sequence $(\alpha_n)$ such that
   \begin{equation*}
   \mu=\sum\nolimits_n\mu_{G{\alpha_n}}+ \mu_0
   \end{equation*}
    the sum being norm convergent; \item[(ii)] $\mu_{G_\alpha}$
and $\mu_0$ are in $\aac^\perp$ provided so is $\mu$;

    \item[(iii)] $\mu_{G_\alpha}\!\perp\mu_{G_\beta}=0$ if
$\alpha\neq\beta$ and $\mu_{G_\alpha}\!\perp\mu_{G_0}=0$ for all
$\alpha$.
    %
   \end{enumerate}
   \end{thm}
    What relates Gleason parts to pick points is (see
    \cite[ChapterVI, Theorem 3.1]{gam}).
   \begin{thm} \tlabel{t2.26.11}
   If $\{x\}$, $x\in X$, is a peak set for $\rrc(X)$ then the
   Gleason part which contains $x$ has a positive planar Lebesque
   measure.
   \end{thm}

   \section*{Representations of $\rrc(X)$ and their elementary spectral measure}

   An algebra homomorphism $\varPhi$ of $\rrc(X)$ into
$\bbs(\hhc)$, the algebra of all bounded linear operators on a
Hilbert space $\hhc$, is called a {\em representation of $\rrc(X)$
on $\hhc$} if it is bounded and $\varPhi(1)=I$. It follows from
the Hahn--Banach theorem and Riesz representation theorem that for
every $f,g\in\hhc$ there a complex measure $\mu_{f,g}$ such that
   \begin{gather*}
   \is{\varPhi(u)f}g=\int_X u\D \mu_{f,g},\quad u\in\rrc(X),
    \\
   \|\mu_{f,g}\|\Le\|\varPhi\|\|f\|\|g\|.
   \end{gather*}
   Call any system $\{\mu_{f,g}\}_{f,g\in \hhc}$ that of {\em
elementary spectral measure of $\varPhi$}; we refer also to the
system $\{\mu_{f,g}\}_{f,g\in \hhc}$ as elementary measures of the
operator
   \begin{equation} \label{4.29.11}
   \text{ $T\okr\varPhi(u_1)$ where $u_1(z)=z$ on $X$. }
   \end{equation}

   With notation $u_1(z)=z$ one can show immediately that the
spectrum $\spek{T}$ of $T\okr \varPhi(u_1)$ is contained in $X$.

   Given a bounded projection (=idempotent) $Q$ on $\mmc(X)$, the
dual of $\ccc(X)$. We say that $Q$ has the {\em property R}, after
F. and M. Riesz, if
   \begin{gather} \notag
   \mu\in \rrc(X)^\perp \;\Longrightarrow\; Q\mu\in\rrc(X)^\perp
\\\label{2.29.10}
   uQ\mu=Q(u\mu),\quad u\in\ccc(X),\;\mu\in\mmc(X).
   \end{gather}

   The important observation, see \cite{studia}, p. 102 and more
in Section 3 of \cite{pre}, is that for a system $\qqs$ of
commuting projections having the property R so does any member of
the Boolean algebra $\bbc(\qqs)$ of projections the system
generates.

   For a system $\{\mu_{f,g}\}_{f,g}$ of elementary spectral
measures of a representation $\varPhi$ one may try to define a
representation $\varPhi_Q$ by
   \begin{equation*}
   \is{\varPhi_Q(u)f}g\okr \int_Xu\D Q\mu_{f,g},\quad
u\in\rrc(X),\;f,g\in\hhc.
   \end{equation*}
   If the projection $Q$ has the property R, then $\varPhi_Q$ is
uniquely determined and $P_Q\okr \varPhi_Q(1)$ is a projection in
$\bbs(\hhc)$ such that
   \begin{gather*}
   P_Q\varPhi(u)=\varPhi(u)P_Q=\varPhi_Q(u),\quad u\in\rrc(X);
   \\
   \text{$\varPhi_Q$ is a representation of $\rrc(X)$ on
$\hhc_Q\okr P_Q\hhc$;}
   \\
   \text{$P_{Q_1}P_{Q_2}=P_{Q_2}P_{Q_1}=0$ if $Q_{1}$ and $Q_2$
are such with ${Q_1}{Q_2}={Q_2}{Q_1}=0$};
   \\
   \text{if $\|\varPhi\|=1$ the projections $P_Q$ are orthogonal.}
   \end{gather*}
   Call $\varPhi_Q$ the $Q$--{\em part of $\varPhi$}; this
definition applies to the operator $T=\varPhi(u_1)$ as well.

   Referring to the above let us restate Theorem 5.3 of
\cite{studia} as follows\,\footnote{\;These results have passed
unnoticed because, presumably, the people recognized in the area
have not found it deserving
 any attention; even in so authoritative monograph like
\cite{paul} there is no mention of it though similarity is one of
the leading topics therein; {\em winner takes all!} }

   \begin{thm} \tlabel{t1.01.11}
   Let $\varPhi$ be a representation of $\rrc(X)$ on $\hhc$.
Suppose $\qqs$ is a system of commuting projections having the
property R and $\{Q_\alpha\}_\alpha\subset\bbc(\qqs)$ is composed
of projections such that $Q_\alpha Q_\beta=0$ for
$\alpha\neq\beta$ then there exists $S\in\bbs(\hhc)$ with
$S\odw\in\bbs(\hhc)$ such that
   \begin{equation} \label{2.01.11}
S\odw\varPhi(u)S=\bigoplus_\alpha\varPhi_\alpha(u)\oplus\varPhi_0(u),\quad
u\in\rrc(X),
   \end{equation}
   where $\varPhi_\alpha$ is the $Q_\alpha$--part of $\varPhi$ and
$\varPhi_0$ is the ${\bigwedge_\alpha (I-Q_\alpha)}$--part of
$\varPhi$ {\rm (}the latter refers to the Boolean operation in
$\bbc(\qqs)$$)$.
   \end{thm}
   When $\varPhi$ is contractive, all the projections $P_Q$ become
contractive as well and there is no need to look for any
similarity $S$; this was developed in \cite{mlak} where a dilation
free extension of results of \cite{sarason} is treated, for some
application to subnormal operators see \cite{lau}. There is one
more instance when the same effect appears, see \cite{bull}.
   \begin{thm} \tlabel{t1.04.11}
   Under the assumptions of {\rm(Theorem \ref{t1.01.11})} the
operator $S$ appearing there can be chosen to be the identity
operator provided there exists a system $\{\mu_{f,g}\}_{f,g}$ of
elementary spectral measures of $\varPhi$ such that
   \begin{gather} \label{2.04.11}
   \text{the measures $\mu_{f,f}$, $f\in\hhc$, are real;}
   \\\label{3.04.11}
   \text{$Q\mu_{f,f}$ are real too for any $Q$ in question.}
   \end{gather}

   \end{thm}
    It is clear that the above happens because, due to
\eqref{2.04.11} and \eqref{3.04.11}, all the projections $P_Q$
involved become \underline{selfadjoint}. Notice also that, under
these circumstances the result may be applicable to, the condition
\eqref{3.04.11} is automatically satisfied.


   \subsection*{Spectral sets and their representations}
Given $K>0$, call a compact set $X\subset\ccb$ a {\em
$K$--spectral set} of $T\in\bbs(\hhc)$
   \begin{equation*}
   \|u(T)\|\Le K\sup\nolimits_{z\in X}|u(z)|,\quad\text{for all
rational functions $u$ with poles off $X$.}
   \end{equation*}
   If $K$ can be chosen to be $1$ call $X$ a {\em von Neumann
spectral set} of $T$. Moreover, if for $X$ and $T$ there is $K$
such that $X$ is a $K$--spectral set of $T$, we say that $X$ is
just a {\em spectral set} of $T$\,\footnote{\;Please notice a
little deviation from the standard terminology.}.

   Von Neumann's celebrated theorem stays that the unit disc is a
von Neumann spectral set for a contraction. In fact, positive
results in the matter compete with negative ones; mostly because
the number one candidate as spectrum of an operator is, with all
its possible oddities like holes, gives a real trouble. It seems
that
 spectral sets (or one may prefer $K$-spectral sets) in fact
involve two parameters $X$ and $K$ and this opens the doors to
some activity, which has been done for long, including for those
who like to optimize.
   \subsection*{The Delyons result and its adherents}
   One of the results we have just had in mind is this \cite{dd}
which follows. It is in a sense far going\,\footnote{\;Except that
in \cite{put}, cf.Theorem \ref{t2.28.11}, we ought to mention also
(some of) other which come up: \cite{cr1} as well as \cite{cr2},
\cite{bad} and \cite{beck1}.}.
   \begin{thm} \tlabel{t3.24.10}
   Let $T$ be an operator on a Hilbert space and $X$ be a bounded
convex subset of $\ccb$ containing $\wwf(T )$. Then $X$ is a
spectral set of $T$.
   \end{thm}

 \begin{thm} \tlabel{t4.24.10}
   Let $X$be as in {\rm Theorem \ref{t3.24.10}} and assume that is
has a piecewise $\ccc^1$ boundary $\partial X$. Denote by
$\ccc(\partial X)$ the space of continuous functions on $\partial
X$ endowed with the uniform norm. Under the assumptions of {\rm
Theorem \ref{t3.24.10}}, there exist a continuous linear operator
$S$ on $\ccc(\partial X)$ and a semispectral measure $F$ on
$\partial X$ such that
   \begin{equation*}
   u(T)=\int_{\partial X}S u \D F
   \end{equation*}
for any rational function $u$ with poles off $X$.
   \end{thm}
   A few years later the above theorem was subsumed by Putinar and
Sandberg \cite{put} into the double layer potential theory of Carl
Neumann, which resulted in a kind of normal dilation theorem. The
latter
 can be read as follows.
   \begin{thm} \label{t2.28.11}
   Let $T$ be a bounded linear operator in a Hilbert space $\hhc$
and let $X$ be a compact convex set which contains the numerical
range of T . Then there exists a normal operator $N\in\bbs(\kkc)$,
acting on a larger space $\kkc$, with spectrum on the curve
$\partial X$, such that:
   \begin{equation} \label{1.29.11}
   u(T ) = 2P[(I + K)^{-1}u](N)P
   \end{equation}
for any function $u$ continuous on $X$ and harmonic in the
interior of $X$. Here $P$ is the orthogonal projection of $\kkc$
onto $\hhc$ and the linear continuous transformation $\funk K
{\ccc(\partial X)} {\ccc(\partial X)}$ is the classical
Neumann-Poincare singular integral operator.
   \end{thm}
   Therefore, the relation between Theorem \ref{t4.24.10} and
Theorem \ref{t2.28.11} is in $$S=2 [(I+K)^{-1}].$$ We set up
Theorem \ref{t2.28.11} into further inquiry ending in a little
lemma. For any $f\in\hhc$ there is a unique measure $\mu_f$ such
that
   \begin{equation*}
   \int_{\partial X}u\D \mu_f =\is{2P[(I + K)^{-1}u](N)P f}f.
   \end{equation*}

   \begin{lem} \tlabel{t1.29.11}
   The measures $\mu_f$, $f\in\hhc$ are real.
   \end{lem}
   \begin{proof}
   Just a couple of words for the proof: as the Neumann-Poincar\'e
integral operator of $\ccc(\partial X)$ into itself, which really
$K$ is, maps real functions into real themselves, the whole
argument with Neumann series for $(I+K)^{-1}$ from the very bottom
of p. 348 of \cite{put} (after Theorem 1 therein) applied to a
real $u$ makes therefore the measures $\mu_f$ real.
   \end{proof}
   \begin{rem} \tlabel{t1.11.2}
   In \cite{bull} a generalization of $\rho$--dilatability was
proposed, very much in flavour of \eqref{1.29.11}. It was designed
to generate real elementary measures like $\ccc_\rho$ operators
do.
   \end{rem}

   \subsection*{The orthogonal decomposition}

Notice that either from Theorem \ref{t3.24.10} or from Theorem
\ref{t2.28.11}, depending on a kind of assumption on $X$ on wants
to impose, it follows that $X$ is a spectral set of $T$.
Consequently $T$ generates a representation, say $\varPhi$, of
$\rrc(X)$, according to \eqref{4.29.11}.

   Now we are in a position to state our decomposition result.
   \begin{thm} \tlabel{t3.29.11}
   Suppose $T\in\bbs(\hhc)$ and $X\subset\ccb$ are as in {\rm
Theorem \ref{t2.28.11}} and there is a system of elementary
spectral measures of $T$ such that \eqref{2.04.11} holds. Suppose
$\qqs$ is a system of commuting projections having the property R,
satisfying \eqref{2.29.10} and \eqref{3.04.11} and
$\{Q_\alpha\}_\alpha\subset\bbc(\qqs)$ is composed of projections
such that $Q_\alpha Q_\beta=0$ for $\alpha\neq\beta$ then there
exists a system $\{S_\alpha\}_\alpha\cup \{S_0\}$ of similarities
in $\hhc_\alpha\okr P_\alpha\hhc$ and $\hhc_0\okr P_0\hhc$ such
that
   \begin{equation} \label{2.01.11}
\varPhi(u)=\bigoplus_\alpha
S_\alpha\odw\varPsi_\alpha(u)S_\alpha\oplus
S_0\odw\varPsi_0(u)S_0,\quad u\in\rrc(X),
   \end{equation}
   where $S_\alpha\odw\varPsi_\alpha S_\alpha$ is the
$Q_\alpha$--part of $\varPhi$ and $S_0\odw\varPhi_0 S_0$ is the
${\bigwedge_\alpha (I-Q_\alpha)}$--part of $\varPhi$. The
representations $\varPsi_\alpha$ and $\varPsi_0$ are contractive.
   \end{thm}
   \begin{proof}
   The only thing which requires some explanation is appearance of
the similarities $S_\alpha$ and $S_0$. The representations
$\varPsi_\alpha$ and $\varPsi_0$ which can be get from
\eqref{2.01.11} with $S=I$ according to Theorem \ref{t1.04.11}, or
rather the corresponding operators $T_\alpha=\varPsi_\alpha(u_1)$
and $T_0=\varPsi_0(u_1)$ have their numerical ranges contained in
$X$ as well; this is due to \eqref{1.23.10}. Theorem 2 of
\cite{cr1} tells us that each of those
$T_\alpha=\varPsi_\alpha(u_1)$'s as well as $T_0=\varPsi_0(u_1)$
are completely bounded. Now an application of Theorem 9.1, p. 120
of \cite{paul} generates the similarities in question.
   \end{proof}
   Notice that one of the new ingredients in Theorem
\ref{t3.29.11} comparing to what is in \cite{bull} is the
appearance of similarities within the orthogonal decomposition, or
in other words, a kind of diagonalization with respect to the
aforesaid orthogonal decomposition.
   \subsection*{Particular cases}  The two particular cases we can
   apply Theorem \ref{t3.29.11} to are the decompositions
determined by those for sets of antisymmetry (Theorem
\ref{1.24.10}) and Gleason parts (Theorem \ref{t1.26.11}). To
state the relevant results for an operator $T$ which generates the
representation $\varPhi$ is a matter of necessity. Let us mention
only that the assumptions of
   \eqref{2.29.10} and \eqref{3.04.11} to hold can be removed due
to the nature of projections $Q$ involved.

   \bibliographystyle{amsplain}
   
   \end{document}